\documentclass[11pt]{article}

\usepackage{amsmath,amssymb,amsfonts,amsthm,enumerate,thmtools,hyperref, parskip,tikz,ifthen,etoolbox}
\usetikzlibrary{positioning,chains,fit,shapes,calc,decorations.pathreplacing,backgrounds}

\usepackage[symbol]{footmisc}

\newcommand{\eps}{\varepsilon}
\newtheorem{theorem}{Theorem}
\newtheorem{lemma}[theorem]{Lemma}

\newtheorem{definition}[theorem]{Definition}

\newtheorem{corollary}[theorem]{Corollary}
\newtheorem{claim}[theorem]{Claim}
\newtheorem{example}[theorem]{Example}
\newtheorem{proposition}[theorem]{Proposition}

\newtheorem{observation}[theorem]{Observation}

\renewcommand{\phi}{\varphi}

\usepackage[font=small,labelfont=bf]{caption}

\newcommand{\grstar}[3]{%
	\ifthenelse{#1>1}
	{ %then
		\node at (#3+#1/2+0.5,2) (o#2){};
		\foreach \xitem in {1,2,...,#1}
		{%
			\node at (#3+\xitem,0) (#2\xitem) {};
			\draw (#2\xitem) -- (o#2);
		}
	}
	{ %else
		\node at (#3+#1/2+0.5,2) (o#2){};
		\node at (#3+#1/2+0.5,0) (1#2){};
		\draw (o#2)--(1#2);	
	}
}

\title{On embedding degree sequences}
\author{
	B\'ela Csaba\thanks{Bolyai Institute, Interdisciplinary Excellence Centre, University of 
Szeged, Hungary, {email: bcsaba@math.u-szeged.hu}. Partially supported by Ministry of Human Capacities, Hungary, Grant 20391-3/2018/FEKUSTRAT, NKFIH Fund No.~KH 129597 and by the NKFIH Fund No.~SNN 117879.},
	B\'alint V\'as\'arhelyi\thanks{Bolyai Institute, University of Szeged, Hungary, email: mesti@math.u-szeged.hu. {Supported by T\'AMOP-4.2.2.B-15/1/KONV-2015-0006.}} 
}

\begin{document}

\date{}

\maketitle

\abstract{
	Assume that we are given two graphic sequences, $\pi_1$ and $\pi_2$. We consider conditions for $\pi_1$ and $\pi_2$ which guarantee that there exists a simple graph $G_2$ realizing $\pi_2$ such that $G_2$ is the subgraph of any simple graph $G_1$ that realizes $\pi_1$.
}

%%%%%%%%%%%%%%%%%%%%%%%%%%%%%%%%%%%%%%INTRO + PRELIMINARIES%%%%%%%%%%%%%%%%%%%%%%%%%%%%%%%%%%%%%%%%

\section{Introduction}
All graphs considered in this paper are simple. We use standard graph theory notation, see for example~\cite{west}. Let us provide a short list of a few perhaps not so common notions, notations. Given a bipartite graph $G(A, B)$ we call it  \textit{balanced} if 
$|A|=|B|$. This notion naturally generalizes for $r$-partite graphs with $r\in \mathbb{N},$ $r\ge 2$. 

If $S\subset V$ for some graph $G=(V, E)$ then the subgraph spanned by $S$ is denoted by $G[S]$. Moreover, 
let $Q\subset V$ so that $S\cap Q=\emptyset,$ then $G[S, Q]$ denotes the bipartite subgraph of $G$ on vertex
classes $S$ and $Q,$ having every edge of $G$ that connects a vertex of $S$ with a vertex of $Q$. The number
of vertices in $G$ is denoted by $v(G),$ the number of its edges is denoted by $e(G)$. The degree of a vertex $x\in V(G)$ is denoted by $deg_G(x),$ or if $G$ is clear from the context, by $deg(x)$. The number of neighbors of $x$ in a subset 
$S\subset V(G)$ is denoted by $deg_G(x, S),$ and $\delta(G)$ and $\Delta(G)$ denote the minimum and maximum degree of $G,$ respectively. 
%The chromatic number of a graph $G$ is $\chi(G)$. 
The complete graph on $n$ vertices
is denoted by $K_n,$ the complete bipartite graph with vertex class sizes $n$ and $m$ is denoted by $K_{n, m}$.

A finite sequence of natural numbers $\pi=(d_1,\ldots,d_n)$ is a \textit{graphic sequence} or \textit{degree sequence} if there exists a graph $G$ such that $\pi$ is the (not necessarily) monotone degree sequence of $G$. Such a graph $G$ \textit{realizes} $\pi$.
For example, the degree sequence $\pi = (2,2,\ldots,2)$ can be realized only by vertex-disjoint union of cycles.

The largest value of $\pi$ is denoted by $\Delta(\pi)$. Often the positions of $\pi$ will be identified with the elements of a vertex set $V$. In this case, we write $\pi(v)$ $(v\in V)$ for the corresponding component of $\pi$.

The degree sequence $\pi=(a_1,\ldots,a_k;\allowbreak b_1,\ldots,b_l)$ is a \textit{bigraphic} sequence if there exists a simple bipartite graph $G = G(A,B)$ with $|A| = k$, $|B| = l$ realizing $\pi$
such that the degrees of vertices in $A$ are $a_1, \ldots, a_k,$ and the degrees of the vertices of $B$ are $b_1, \ldots, b_l$.

Let $G$ and $H$ be two graphs on $n$ vertices. We say that $H$ is a subgraph of $G,$ if we can delete edges from $G$ so that we obtain an isomorphic copy of $H$. We denote this relation
by $H\subset G$. In the literature the equivalent complementary formulation can be found as well: we say that $H$ and $\overline{G}$ \textit{pack} if there exist edge-disjoint copies of $H$ and $\overline{G}$ 
in $K_n$. Here $\overline{G}$ denotes the \textit{complement} of $G$.
%Two degree sequences $\pi_1$ and $\pi_2$ \textit{pack}, if there are graphs $G_1$ and $G_2$ realizing $\pi_1$ and $\pi_2$, respectively, such that $G_1$ and $G_2$ pack.

Two classical results in this field are the following theorems.

\begin{theorem}[Dirac, \cite{dirac}]
	Every graph $G$ with $n\geq 3$ vertices and minimum degree $\delta(G) \geq \frac{n}{2}$ has a Hamilton cycle.
\end{theorem}

\begin{theorem}[Corr\'adi-Hajnal, \cite{corradi}]
	Let $k\geq 1$, $n\geq 3k$, and let $H$ be an $n$-vertex graph with $\delta(H)\geq 2k$. Then $H$ contains $k$ vertex-disjoint cycles.
\end{theorem}

It is an old an well-understood problem in graph theory to tell whether a given sequence of natural numbers is a degree sequence 
or not. We consider a generalization of it, which is remotely related to the so-called discrete tomography\footnote{In the discrete tomography problem we are given two degree sequences of length $n$, $\pi_1$ and $\pi_2,$ and the questions is whether there exists a graph $G$ on $n$ vertices with a red-blue edge coloration so that the following holds: for every vertex $v$ the red degree of $v$ is $\pi_1(v)$ and the blue degree of $v$ is $\pi_2(v)$.} (or degree sequence packing) problem (see e.g.~\cite{ferrara}) as well.
The question whether a sequence of $n$ numbers $\pi$ is a degree sequence can also be formulated as follows: Does $K_n$ have a subgraph $H$ such that the degree sequence of $H$ 
is $\pi?$ The question becomes more general if $K_n$ is replaced by some (simple) graph $G$ on $n$ vertices. If the answer is yes, we say that $\pi$ \textit{can be embedded into} $G,$ or equivalently, 
$\pi$ packs with $\overline{G}$. 

One of our main results is the following.

\begin{theorem}\label{fotetel}
For every $\eta>0$ and $D\in\mathbb{N}$ there exists an $n_0 = n_0(\eta, D)$ such that for all $n>n_0$ if $G$ is a graph on $n$ vertices with $\delta(G)\geq \left(\frac{1}{2}+\eta\right)n$ and $\pi$ is a degree sequence of length $n$ with $\Delta(\pi) \leq D$, then $\pi$ is embeddable into $G$.
\end{theorem}

%We also state \autoref{fotetel} in an equivalent complementary form, as a packing problem.

%\begin{theorem}
%	For every $\eps>0$ and $D\in\mathbb{N}$ there exists an $n_0 = n_0(\eps, D)$ such that for all $n>n_0$  if $\pi_1$ %and $\pi_2$ are graphic sequences of length $n$ satisfying $\Delta(\pi_1)< \(\frac{1}{2}-\eps\)n$ and $\Delta(\pi_2)\leq D$ %then there exists a graph $G_2$ that realizes $\pi_2$ and packs with any $G_1$ realizing $\pi_1$.
%	\label{th2}
%\end{theorem}

It is easy to see that \autoref{fotetel} is sharp up to the $\eta n$ additive term. For that let $n$ be an even number, and suppose that every element of $\pi$ is 1. Then the only graph that realizes 
$\pi$ is the union of $n/2$ vertex disjoint edges. Let $G=K_{n/2-1, n/2+1}$ be the complete bipartite graph with vertex class sizes $n/2-1$ and $n/2+1$. Clearly $G$ does not have $n/2$ vertex disjoint edges.

In order to state the other main result of the paper we introduce a new notion. 

\begin{definition}
Let $q\ge 1$ be an integer.
A bipartite graph $H$ with vertex classes $S$ and $T$ is $q$-unbalanced, if $q|S|\le |T|$. The degree sequence $\pi$ is $q$-unbalanced, if it can be realized by a $q$-unbalanced bipartite graph.  
\end{definition}

%In another direction, one can also show that if $\pi$ has little less elements than the number of vertices in $G,$ then
%$\pi$ can be embedded into $G$ under very similar conditions.

\begin{theorem}\label{konstans}

	Let $q\ge 1$ be an integer. For every $\eta>0$ and $D\in \mathbb{N}$ there exist an $n_0=n_0(\eta, q)$ and an $M=M(\eta, D, q)$ such that if 
	$n\ge n_0,$ $\pi$ is a $q$-unbalanced degree sequence of length $n-M$
	with $\Delta(\pi)\le D,$ $G$ is a graph on $n$ vertices with $\delta(G)\geq \frac{n}{q+1}+\eta n$, then $\pi$ can be embedded into $G$.
\end{theorem}

Hence, if $\pi$ is unbalanced, the minimum degree requirement of \autoref{fotetel} can be substantially decreased, what we pay for this is that the length of $\pi$ has to be slightly smaller than the number 
of vertices in the host graph.

\section{Proof of  \autoref{fotetel}}

\begin{proof}
	
	First, we find a suitable realization $H$ of $\pi,$ our $H$ will consists of components of bounded size. Second, we embed $H$ into $G$ using a theorem by Chv\'atal and Szemer\'edi, and a result on embedding so called well-separable graphs.
The details are as follows.

%In order to realize $\pi$ by $H$ we will use two types of ''gadgets''. Type 1 gadgets are balanced complete bipartite graphs on $2k$ vertices, where $k\in\{1,\ldots,\Delta(\pi)\}$, these are the components of $H[V-A]$. Type 2 gadgets are composed of at least two type 1 gadgets and at most two other vertices, these are the components of $H[A]$.

%shown on \autoref{fig1}.

We construct $H$ in several steps.
At the beginning, let $H$ be the empty graph and let all degrees in $\pi$ be \textit{active}.	
While we can find $2i$ active degrees of $\pi$ with value $i$ (for some $1\leq i\leq \Delta(\pi)$) we realize them with a $K_{i,i}$ (that is, we add this complete bipartite graph to $H$, and the $2i$ degrees are ``inactivated''). When we stop we have at most
$\sum_{i=1}^{\Delta(\pi)} (2i-1)$ active degrees.
This way we obtain several components, each being a balanced complete bipartite graph. These are the \textit{type 1 gadgets}. Observe that if a vertex $v$ belongs to some type 1 gadget, then its degree is exactly $\pi(v).$

Let $R=R_{odd}\cup R_{even}$ be the vertex set that is identified with the active vertices ($v\in R_{odd}$ if and only if the assigned active degree is odd).
Since $\sum\limits_{v\in R}d(v)$ must be an even number we have that $|R_{odd}|$ is even. Add a perfect matching on $R_{odd}$ to $H$. With this we achieved that every vertex of $R$ misses an even number of edges.
 
Next we construct the \textit{type 2 gadgets} using the following algorithm. 
In the beginning every type 1 gadget is \textit{unmarked}. Suppose that $v\in R$ is an active vertex. Take a type 1 gadget
 $K,$ \textit{mark} it, and let $M_K$ denote an arbitrarily chosen perfect matching in
$K$ ($M_K$ exists since $K$ is a balanced complete bipartite graph). Let $xy$ be an arbitrary edge in $M_K$. Delete the $xy$ edge and add the new edges $vx$ and $vy$. While $v$
is missing edges repeat the above procedure with edges of $M_K,$ until $M_K$ becomes empty. If $M_K$ becomes empty, take a new unmarked type 1 gadget $L,$ and repeat the
method with $L$. It is easy to see that in $\pi(v)/2$ steps $v$ reaches its desired degree and gets inactivated. Clearly, the degrees of vertices in the marked type 1 gadgets have not changed. 

	\begin{figure}[h!]
	\centering
\tikzstyle{background rectangle}=
[draw=black,rounded corners=1ex]

\scalebox{0.5}{
	\begin{tikzpicture}[thick,
	snode/.style={draw,fill=black,circle,scale=2},
	every node/.style={inner sep=1pt},
	]
	
	\begin{scope}[xshift=0cm,yshift=0cm,start chain=going right,node distance=5mm]
	\foreach \i in {1,...,7}
	\node[snode,on chain,label=above:1] (a\i) {};
	\end{scope}
	
	\begin{scope}[xshift=0cm,yshift=-4cm,start chain=going right,node distance=5mm]
	\foreach \i in {1,...,7}
	\node[snode,on chain,label=below:2] (b\i) {};
	\end{scope}
	
	\begin{scope}[xshift=9cm,yshift=0cm,start chain=going right,node distance=5mm]
	\foreach \i in {1,...,5}
	\node[snode,on chain,label=above:2] (g\i) {};
	\end{scope}
	
	\begin{scope}[xshift=9cm,yshift=-4cm,start chain=going right,node distance=5mm]
	\foreach \i in {1,...,5}
	\node[snode,on chain,label=below:3] (h\i) {};
	\end{scope}
	
	\node [snode,xshift=3cm,yshift=-1cm,label=below:3] (j1) {};
	\node [snode,xshift=3.7cm,yshift=-1cm,label=below:1] (j2) {};
	
	\draw (j1) -- (j2);
	
	\foreach \i in {1,...,7}
	{	
		\ifthenelse{\i<6}{	\draw (j1)--(a\i); \draw (j1)--(b\i);\draw (j2)--(g\i);\draw (j2)--(h\i);}{\draw (a\i)--(b\i);}
		\foreach \j in {1,...,7}
		{
			\ifthenelse{\not{\j=\i}}{\draw (a\i)--(b\j);}{}
			\ifthenelse{{\i<6}\and{\j<6}\and\not{\j=\i}}{\draw (g\i)--(h\j);}{}
		}
	}
	
	\node [snode,xshift=3.5cm,yshift=-4cm,label=above:1] (z) {};
	
	\begin{scope}[xshift=0cm,yshift=-10cm,start chain=going right,node distance=5mm]
	\foreach \i in {1,...,7}
	\node[snode,on chain,label=below:3] (c\i) {};
	\end{scope}
	
	\begin{scope}[xshift=0cm,yshift=-6cm,start chain=going right,node distance=5mm]
	\foreach \i in {1,...,7}
	\node[snode,on chain,label=above:2] (d\i) {};
	\end{scope}
	
	\begin{scope}[xshift=9cm,yshift=-10cm,start chain=going right,node distance=5mm]
	\foreach \i in {1,...,5}
	\node[snode,on chain,label=below:2] (e\i) {};
	\end{scope}
	
	\begin{scope}[xshift=9cm,yshift=-6cm,start chain=going right,node distance=5mm]
	\foreach \i in {1,...,5}
	\node[snode,on chain,label=above:3] (f\i) {};
	\end{scope}

	\begin{scope}[xshift=4.8cm,yshift=-11.5cm,start chain=going below,node distance=5mm]
	\foreach \i in {1,...,5}
	\node[snode,on chain,label=left:2] (r\i) {};
	\end{scope}
	
	\begin{scope}[xshift=8.8cm,yshift=-11.5cm,start chain=going below,node distance=5mm]
	\foreach \i in {1,...,5}
	\node[snode,on chain,label=right:3] (s\i) {};
	\end{scope}

	%!cd ef rs z
	\draw (j1) -- (j2);
	
	\foreach \i in {1,...,7}
	{	
		\ifthenelse{\i<6}{	\draw (z)--(c\i); \draw (z)--(d\i);\draw (z)--(e\i); \draw (z)--(f\i);\draw (z)--(r\i);\draw (z)--(s\i);}{\draw (c\i)--(d\i);}
		\foreach \j in {1,...,7}
		{
			\ifthenelse{\not{\j=\i}}{\draw (c\i)--(d\j);}{}
			\ifthenelse{{\i<6}\and{\j<6}\and\not{\j=\i}}{\draw (e\i)--(f\j);\draw (r\i)--(s\j);}{}
		}
	}

	\end{tikzpicture}
}
\footnotesize\caption{Type 2 gadgets of $H$ with a 3-coloring}
\protect\label{fig1}
\normalsize

\end{figure}
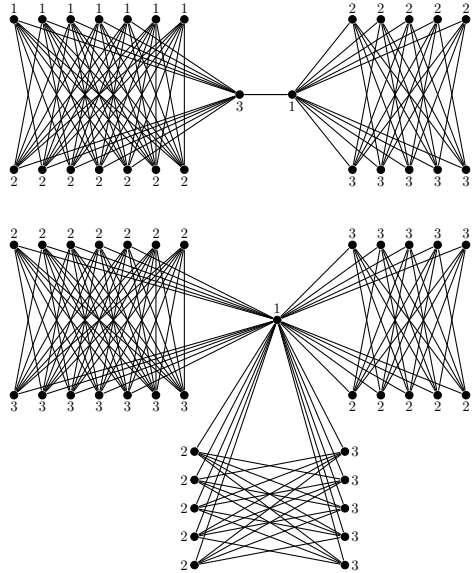

\autoref{fig1} shows examples of type 2 gadgets. In the upper one two vertices of $R_{odd}$ were first connected by an
edge and then two type 1 gadgets were used so that they could reach their desired degree, while in the lower one we used 
three type 1 gadgets for a vertex of $R.$ The numbers at the vertices indicate the colors in the 3-coloring of $H.$

Let $A\subset V(H)$ denote the set of vertices containing the union of all type 2 gadgets. Observe that type 2 gadgets are 
3-colorable and all have less than $5\Delta^2(\pi)$ vertices.  
Let us summarize our knowledge about $H$ for later reference.
	
\begin{claim}\label{HSzerkezet}
\begin{enumerate}[(1)]

		\item $|A|\le 5\Delta^3(\pi),$
		\item the components of $H[V-A]$ are balanced complete bipartite graphs, each having size at most 
			$2\Delta(\pi),$
		\item $\chi(H[A])\le 3,$ and
		\item $e(H[A, V-A])=0$.	
\end{enumerate}
\end{claim}

We are going to show that $H\subset G$. For that we first embed the possibly 3-chromatic part $H[A]$ using the following strengthening of the Erd\H{o}s--Stone theorem proved by Chv\'{a}tal and Szemer\'edi~\cite{chvatal}.
	
	\begin{theorem}
		\label{es}
		Let $\phi>0$ and assume that $G$ is a graph on $n$ vertices where $n$ is sufficiently large. Let $r\in \mathbb{N},$ $r\ge 2$. If \[e(G)\ge \left(\frac{r-2}{2(r-1)}+\phi\right)n^2,\]
		then $G$ contains a $K_r(t)$, i.e. a complete $r$-partite graph with $t$ vertices in each class, such that
		\begin{equation}
			t>\frac{\log n}{500\log\frac{1}{\phi}}.
		\end{equation}
		%\myqed
	\end{theorem}
	
Since $\delta(G)\ge n/2 +\eta n,$ the conditions of \autoref{es} are satisfied with $r=3$ and $\phi=\eta/2,$ hence, $G$ contains a balanced complete tripartite subgraph $T$
on $\Omega(\log n)$ vertices. Using Claim~\ref{HSzerk} and the 3-colorability of $H[A]$ this implies that $H[A]\subset T$.

Observe that after embedding $H[A]$ into $G$ every uncovered vertex of $G$ still has at least $\delta(G)-v(F)> (1/2+\eta/2)n$ uncovered neighbors. Denoting the subgraph of the uncovered vertices of $G$ by $G'$ we obtain that $\delta(G')>(1/2+\eta/2)n.$

In order to prove that $H[V-A]\subset G'$ we first need a definition.
\begin{definition}
	A graph $F$ on $n$ vertices is \emph{well-separable} if it has a subset $S\subset V(F)$ of size $o(n)$ such that all components of $F-S$ are of size $o(n)$. %\myqed
\end{definition}
	
We need the following theorem.

\begin{theorem}[\cite{wellsep}]
	For every $\gamma>0$ and positive integer $D$ there exists an $n_0$ such that for all $n>n_0$ if $F$ is a bipartite well-separable graph on $n$ vertices, $\Delta(F)\le D$ and $\delta(G)\ge \left(\frac{1}{2}+\gamma\right)n$ for a graph $G$ of order $n$, then $F\subset G$.
	\label{th_well}
	%\myqed
\end{theorem}

Since $H[V-A]$ has bounded size components, we can apply \autoref{th_well} for $H[V-A]$ and $G',$ with parameter $\gamma=\eta/2$. With this we finished proving what was desired.

\end{proof}

%%%%%%%%%%%%%%%%%%%%%%%%%%%%%%%%%%%%

\section{Further tools for \autoref{konstans} }

When proving \autoref{fotetel}, we used the Regularity Lemma of Szemer\'edi, but implicitly, via the result on embedding well-separable graphs. When proving
\autoref{konstans} we will apply this very powerful result explicitly, hence, below we give a very brief introduction to the area. The interested
reader may consult with the original paper by Szemer\'edi~\cite{SZ} or e.g. with the survey paper~\cite{KS93}.

\subsection{Regularity Lemma}

The \textit{density}  between disjoint sets $X$ and $Y$ is defined as:
\[d(X,Y)={{e(X,Y)}\over     {|X||Y|}}.\]
We  will need  the following definition to state the Regularity Lemma.

\begin{definition}[Regularity condition] Let $\varepsilon >0$. A pair
  $(A,B)$      of     disjoint      vertex-sets     in      $G$     is
  \mbox{$\varepsilon$-regular}  if  for every  $X  \subset  A$ and  $Y
  \subset B$, satisfying
\[|X|>\varepsilon |A|,\ |Y|>\varepsilon |B|\]
we have
\[|d(X,Y)-d(A,B)|<\varepsilon.\]
\end{definition}
This  definition  implies  that   regular  pairs  are  highly  uniform
bipartite graphs; namely, the  density of any reasonably large subgraph
is almost the same as the density of the regular pair.

We will use the following form of the Regularity Lemma:

\begin{lemma}[Degree Form] \label{rl0} 
For every $\varepsilon>0$ there is an $M=M(\varepsilon)$ such
that if $G=(W,E)$ is any graph and $d\in [0,1]$ is any real number, then there is a partition of the
vertex set $V$ into $\ell+1$ clusters $W_0, W_1,\ldots,W_\ell$, and there is a subgraph $G'$ of $G$
with the following properties:  
\begin{itemize} 
\item $\ell \le M$, 
\item $|W_0|\le \varepsilon |W|$,
\item all clusters $W_i$, $i\ge 1$, are of the same size $m \ \left(\le \left\lfloor {|W|\over \ell}\right\rfloor<\varepsilon |W| \right)$, 
\item $deg_{G'}(v)>deg_{G}(v)-(d+\varepsilon)|W|$ for all $v\in W$, 
\item $G'|_{W_i}=\emptyset$ ($W_i$ is an independent set in $G'$) for all $i\ge 1$, 
\item all pairs
$(W_i,W_j)$, $1\le i <j \le \ell$, are $\varepsilon$-regular, each with density either 0 or greater
than $d$ in $G'$.  
\end{itemize} 
\end{lemma}

\noindent We  call  $W_0$  the  \textit{exceptional  cluster}, $W_1, \ldots, W_{\ell}$ are the \textit{non-exceptional clusters}.
In the rest of the paper we will assume that $0< \varepsilon\ll  d \ll  1$. Here $a\ll b$ means that \textit{$a$ is sufficiently smaller than $b$.}

\begin{definition}[Reduced graph]
Apply \autoref{rl0} to the graph $G=(W,E)$ with parameters $\varepsilon$ and $d$, 
and denote the clusters of the resulting partition by $W_0, W_1,\ldots, W_\ell$ ($W_0$ being the exceptional cluster). We construct a new graph $G_r$, the reduced graph of $G'$ in
the following way: The non-exceptional clusters of $G'$
are the vertices of the reduced graph $G_r$ (hence $v(G_r)=\ell$). We connect two vertices of
$G_r$ by an edge if the corresponding two clusters form an $\varepsilon$-regular
pair with density at least $d$.
\end{definition}

The following corollary is immediate:

\begin{corollary}\label{redukaltfok}
Apply \autoref{rl0} with parameters $\varepsilon$ and $d$ to the graph $G=(W,E)$ satisfying 
$\delta(G) \ge \gamma n$ $\ (v(G)=n)$ for some $\gamma>0$. 
Denote $G_r$ the reduced graph of $G'$. Then $\delta(G_r) \ge (\gamma - \theta)\ell$,
where $\theta=2\varepsilon+d$.
\end{corollary} 

The (fairly easy) proof of the lemma below can be found in~\cite{KS93}.

\begin{lemma} \label{felezo}
Let $(A,B)$ be an \mbox{$\varepsilon$-regular}--pair with density $d$ for some $\epsilon >0$. Let $c>0$ be a constant such that $\varepsilon\ll c$.
We arbitrarily divide $A$ and $B$ into two parts, obtaining the non-empty subsets $A',A''$ and $B',B''$, respectively. 
Assume that $|A'|, |A''|\ge c|A|$ and $|B'|, |B''|\ge c|B|$. Then the 
pairs $(A',B'),$ $(A', B''),$ $(A'', B')$ and $(A'',B'')$ are all $\varepsilon/c$--regular pairs with density at least $d-\varepsilon/c$.
\end{lemma}

\subsection{Blow-up Lemma}

Let $H$ and $G$ be two graphs on $n$ vertices. Assume that we want to find an isomorphic copy of $H$ in $G$.
In order to achieve this one can apply a very powerful tool, the Blow-up Lemma of Koml\'os, S\'ark\"ozy and 
Szemer\'edi~\cite{KSSz97,KSSz98}.
For stating it we need a new notion, a stronger one-sided property of regular pairs.

\begin{definition}[Super-Regularity condition] Given a graph $G$ and
  two disjoint subsets  of its vertices $A$ and  $B$, the pair $(A,B)$
  is   \mbox{$(\varepsilon, \delta)${\rm-super-regular}},   if   it   is
  $\varepsilon$-regular and furthermore,
$$deg(a)>\delta |B|,{\rm \ for\  all}\  a\in A,$$
and
$$deg(b)>\delta |A|,{\rm \ for\  all}\  b\in B.$$
\end{definition}

\begin{theorem}[Blow-up Lemma]\label{blow-up}
Given a graph $R$ of order $r$ and positive integers $\delta, \Delta,$ there exists
a positive $\eps=\eps(\delta, \Delta, r)$
such that the following holds: 
Let $n_1,n_2,\ldots,n_r$ 
be arbitrary positive 
parameters and let
us replace the vertices $v_1,v_2,\ldots,v_r$ of $R$ with pairwise disjoint sets $W_1,W_2, \ldots, W_r$ of sizes 
$n_1,n_2,\ldots,n_r$ (blowing up $R$). 
We construct two graphs on the same vertex set $V=\cup_i W_i$.  The first graph $F$ is obtained by replacing each edge 
$v_iv_j \in E(R)$ with the complete bipartite graph between $W_i$ and $W_j$. A sparser graph $G$ is constructed by replacing each edge 
$v_iv_j$ arbitrarily with an $(\eps, \delta)$-super-regular pair between $W_i$ and $W_j$. If a graph $H$ with $\Delta(H)\le \Delta$ is embeddable
into $F$ then it is already embeddable into $G$.
\end{theorem}

\section{Proof of \autoref{konstans}}

Let us give a brief sketch first.
Recall, that $\pi$ is a $q$-unbalanced and bounded degree sequence with $\Delta(\pi)\le D$. In the proof we first show that there exists a $q$-unbalanced
bipartite graph $H$  that realizes $\pi$ such that $H$ is the vertex disjoint union of the graphs $H_1, \ldots, H_k,$ where each $H_i$ graph is a bipartite $q$-unbalanced graph
having bounded size.
We will apply the Regularity lemma to $G,$ and find a special substructure (a decomposition into vertex-disjoint stars) in the reduced graph of $G$. This substructure can then be used to embed 
the union of the $H_i$ graphs, for the majority of them we use the Blow-up lemma.

\subsection{Finding $H$}

The goal of this subsection is to prove the lemma below.
 
{%Kulcsfontosságú lemma bizonyítása
\begin{lemma}
	Let $\pi$ be a $q$-unbalanced degree sequence of positive integers with $\Delta(\pi)\le D$. Then $\pi$ can be realized by a $q$-unbalanced bipartite graph $H$ which is the vertex disjoint union of the graphs $H_1, \ldots, H_k,$ such that for every $i$ we have that $H_i$ is $q$-unbalanced, moreover, $v(H_i)\le 4D^2$.
	\label{th2}
\end{lemma}

Before starting the proof of~\autoref{th2}, we list a few necessary notions and results.

We call a finite sequence of integers a \textit{zero-sum sequence} if the sum of its elements is zero. 
The following result of Sahs, Sissokho and Torf plays an important role in the proof of \autoref{th2}. 

\begin{proposition}\cite{sissokho}
	Assume that $K$ is a positive integer. Then any zero-sum sequence on $\{-K,\ldots, K\}$ having length at least $2K$ contains a proper nonempty zero-sum subsequence.
	\label{th_sissokho}
\end{proposition}

The following result, formulated by Gale \cite{gale} and Ryser \cite{ryser}, will also be useful. We present it in the form as discussed in Lov\'asz \cite{lovasz}.

\begin{lemma}\cite{lovasz}
	\label{lovlemma}
	Let $G=(A,B;E(G))$ be a bipartite graph and $f$ be a nonnegative integer function on $A\cup B$ with $f(A) = f(B)$. Then $G$ has a subgraph $F=(A,B;E(F))$ such that $d_F(x)=f(x)$ for all 
          $x\in A\cup B$ if and only if 
	\begin{equation}
		f(X)\leq e(X,Y) + f(\overline Y)
		\label{loveq}
	\end{equation}
	for any $X\subseteq A$ and $Y\subseteq B$, where $\overline Y = B - Y$.
\end{lemma}

We remark that such a subgraph $F$ is also called an $f$-factor of $G$.

\begin{lemma}
	If $f = (a_1,\ldots,a_s;b_1,\ldots,b_{t})$ is a sequence of positive integers with $s,t\ge 2\Delta^2$, where $\Delta$ is the maximum of $f$, and $f(A) = f(B)$ with $A = \{a_1,\ldots,a_s\}$ and $B = \{b_1,\ldots,b_t\}$ then $f$ is bigraphic.
	\label{stetel}
\end{lemma}

\begin{proof}
	All we have to check is whether the conditions of \autoref{lovlemma} are met if $G = K_{s,t}$.
	
	Suppose indirectly that there is an $(X,Y)$ pair for which \eqref{loveq} does not hold. Choose such a pair with minimal $|X|+|Y|$.
	Then $X=\emptyset$ or $Y=\emptyset$ are impossible, as in those cases \eqref{loveq} trivially holds. Hence, $|X|, |Y|\geq 1$.
	Assuming that  \eqref{loveq} does not hold, we have that
	
	\begin{equation}
		f(X)\geq e(X,Y) + f(\overline Y) + 1,
	\end{equation}
	
	which is equivalent to
	
	\begin{equation}
		f(X) \geq |X||Y| + f(\overline Y) + 1,
		\label{ind_eq_1}
	\end{equation}
	as $G$ is a complete bipartite graph. Furthermore, using the minimality of $|X| + |Y|$, we know that
	
	\begin{equation}
		f(X-a) \leq |X-a||Y| + f(\overline Y)
		\label{mid_1}
	\end{equation}
	for any $a\in X$. \eqref{mid_1} is equivalent to
	\begin{equation}
		f(X)-f(a)\leq |X||Y| - |Y| + f(\overline Y).
		\label{ind_eq_2}
	\end{equation}
	From \eqref{ind_eq_1} and \eqref{ind_eq_2} we have
	\begin{equation}
		f(a) - 1 \geq |Y|
	\end{equation}
	for any $a\in X$, which implies
	\begin{equation}
		\Delta> |Y|.
		\label{end_eq_1}
	\end{equation}
	
The same reasoning also implies that $\Delta>|X|$ whenever $(X, Y)$ is a counterexample.
Therefore we only have to verify that \eqref{loveq} holds in case $|X|< \Delta$ and $|Y|< \Delta$. Recall that $f(B)\geq t$, as all elements of $f$ are positive.
Hence, $f(X)\leq \Delta|X|\leq \Delta^2$, and $f(\overline Y) = f(B) - f(Y)\geq t - \Delta^2,$ and we get that
	
	\begin{equation}
		f(X)\le \Delta^2\leq t - \Delta^2\le f(\overline Y)\leq f(\overline Y) + e_G(X,Y)
	\end{equation}
	holds, since $t\geq 2\Delta^2$.
\end{proof}

\begin{proof}\label{proof:th2}(\autoref{th2})
	Assume that $J=\big(S,T;E(J)\big)$ is a $q$-unbalanced bipartite graph realizing $\pi$. Hence, $q|S|\le |T|$. Moreover, $|T|\le D|S|,$ since $\Delta(\pi)\le D$. We form vertex disjoint 
tuples of the form $(s; t_1, \ldots, t_h),$ such that $s\in S,$ $t_i\in T,$ $q\le h\le D,$ and the collection of these tuples contains every vertex of $S\cup T$ exactly once.   
	 We define the \textit{bias} of the tuple as $$\zeta= \pi(t_{1})+\cdots+\pi(t_{h})-\pi(s).$$ Obviously, $-D\leq\zeta \leq D^2$. 
	The conditions of \autoref{th_sissokho} are clearly met with $K= D^2$. Hence, we can form groups of size at most $2D^2$ in which the sums of biases are zero. 
This way we obtain a partition of $(S,T)$ into $q$-unbalanced set pairs which have zero bias. While these sets may be small, we can combine them so that each combined set is of size at least $2D^2$
and has zero bias. By \autoref{stetel} these are bigraphic sequences. The realizations of these small sequences give the graphs $H_1, \ldots, H_k$. It is easy to see that $v(H_i)\le 4D^2$ for 
every $1\le i\le k$. Finally, we let $H=\cup_iH_i$.
\end{proof}
}

\subsection{Decomposing $G_r$}

Let us apply the Regularity lemma with parameters $0<\eps \ll d\ll \eta$.
By \autoref{redukaltfok} we have that $\delta(G_r)\ge \ell/(q+1) +\eta \ell/2$.

Let $h\ge 1$ be an integer. An $h$-{star} is a $K_{1, h}$. The \textit{center} of an $h$-star is the vertex of degree $h,$ the other vertices are the \textit{leaves}. In case $h=1$ we pick one of the vertices of the 1-star
arbitrarily to be the center. 
%When we refer to a $t$-star, we enumerate its vertices beginning with the centre.

\begin{lemma}
	The reduced graph $G_r$ has a decomposition $\mathcal{S}$ into vertex disjoint stars such that each star has at most $q$ leaves.
	\label{decomp_3}
\end{lemma}

\begin{proof}
	Take a partial star-decomposition of $G_r$ that is as large as possible. Assume that there are uncovered vertices in $G_r$.  Let $U$ denote those vertices that are covered (so we assume that $U$ has maximal cardinality), and let $v$ be an uncovered vertex.  Observe that $v$ has neighbours only in $U,$ otherwise, if $uv\in E(G_r)$ with $u\notin U$, then we can simply add $uv$ to the star-decomposition, contradicting to the maximality of $U$. See \autoref{abra2} for the possible neighbors of $v$.
	
	\begin{enumerate}[a)]
		\item If $v$ is connected to a 1-star, then we can replace it with a 2-star.
		
		\item If $v$ is connected to the center $u$ of an $h$-star, where $h<q$, then we can replace this star with an $h+1$-star by adding the edge $uv$ to the $h$-star.
		
		\item If $v$ is connected to a leaf $u$ of an $h$-star, where $2\le h\leq q$, then replace the star with the edge $uv$ and an $(h-1)$-star (i.e., delete $u$ from it).

	\end{enumerate}

We have not yet considered one more case: when $v$ is connected to the center of a $q$-star. However, simple calculation shows that for every vertex $v$ at least one of the above three cases must hold, using the minimum degree condition of $G_r$. Hence we can increase the number of covered vertices. We arrived at a contradiction, $G_r$ has the desired star-decomposition.
\end{proof}

\begin{figure}[h!]
	\centering
	\tikzstyle{background rectangle}=
	[draw=black,rounded corners=1ex]
	
	\begin{tikzpicture}
	
	\begin{scope}
	[every node/.style={circle,draw}]
	
	\node at (5.5,4) (v) {v};
	
	\grstar{1}{a}{1};
	\grstar{3}{b}{2};
	\grstar{5}{c}{5};
	\end{scope}
	\begin{scope}
	
	\draw [dashed]  (v) -- (oa) node[left,midway]{$a)$} ;
	\draw [dashed] (v) -- (ob) node[left,midway]{$b)$};
	\draw [dashed] (v) -- (b3) node[right,midway]{$c)$};
	\draw [dashed] (v) -- (oc) node[right,midway]{$d)$};
	
	\end{scope}
	
	\end{tikzpicture}
	
	\footnotesize \caption{An illustration for \autoref{decomp_3}}
	\protect\label{abra2}
	\normalsize
\end{figure}
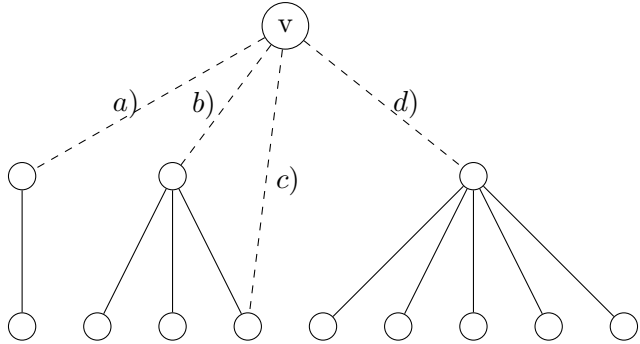

%Let $W_0,W_1,\ldots,W_{\ell}$ be the clusters of $G_R$. 
%$m = |W_1|=\ldots=|W_{}|$. Note that $|W_0|\leq 2d n$.

\subsection{Preparing $G$ for the embedding}
 
Consider the $q$-star-decomposition $\mathcal{S}$ of $G_R$ as in \autoref{decomp_3}. Let $\ell_i$ denote the number of $(i-1)$-stars in the decomposition for every $2\le i\le q+1$.
It is easy to see that $$\sum_{i=2}^{q+1}i\ell_i=\ell.$$ 

First we will make every $\eps$-regular pair in $\mathcal{S}$ super-regular by discarding a few vertices from the non-exceptional clusters. Let for example $\mathcal{C}$ be a star in the
decomposition of $G_r$ with center cluster $A$ and leaves $B_1, \ldots, B_k,$ where $1\le k\le q$. Recall that the $(A, B_i)$ pairs has density at least $d$. 
We repeat the following for every $1\le i\le k:$ if $v\in A$ such that $v$ has at most $2dm/3$ neighbors in $B_i$ then discard $v$ from $A,$ put it into $W_0$. Similarly,
if $w\in B_i$ has at most $2dm/3$ neighbors in $A,$ then discard $w$ from $B_i,$ put it into $W_0$. Repeat this process for every star in $\mathcal{S}$. We have the following:

\begin{claim}\label{szup-reg}
We do not discard more than $q\eps m$ vertices from any non-excepti\-o\-nal cluster. 
\end{claim}  

\begin{proof}
Given a star $\mathcal{C}$ in the decomposition $\mathcal{S}$ assume that its center cluster is $A$ and let $B$ be one of its leaves.  Since the pair $(A, B)$ is $\eps$-regular with density 
at least $d,$ neither $A,$ nor $B$ can have more than $\eps m$ vertices that have at most $2dm/3$ neighbors in the opposite cluster. Hence, during the above process
we may discard up to $q\eps m$ vertices from $A$. Next, we may discard vertices from the leaves, but since no leaf $B$ had more than $\eps m$ vertices with less than
$(d-\eps)m$ neighbors in $A,$ even after discarding at most $q\eps m$ vertices of $A,$ there can be at most $\eps m$ vertices in $B$ that have less than $(d-(q+1)\eps)m$
neighbors in $A$. Using that $\eps\ll d,$ we have that $(d-(q+1)\eps)>2d/3$. We obtained what was desired. 
\end{proof}

\medskip

By the above claim we can make every $\eps$-regular pair in $\mathcal{S}$ a $(2\eps, 2d/3)$-super-regular pair so that we discard only relatively few vertices. Notice that we only have an 
upper bound for the number of discarded vertices, there can be clusters from which we have not put any points into $W_0$. We repeat the following for every non-exceptional 
cluster: if $s$ vertices were discarded from it with $s<q\eps m$ then we take $q\eps m-s$ arbitrary vertices of it, and place them into $W_0$. This way every 
non-exceptional cluster will have the same number of points, precisely $m-q\eps m$. For simpler notation, we will use the letter $m$ for this new cluster size. 
Observe that $W_0$ has increased by $q\eps m \ell$ vertices, but we still have $|W_0|\le 3dn$ since $\eps \ll d$ and $\ell m\le n$. Since $q\eps m\ll d,$ in the resulting pairs
the minimum degree will be at least $dm/2$. 

\medskip

Summarizing, we obtained the following:

\begin{lemma}\label{szupreg}
By discarding a total of at most $q\eps n$ vertices from the non-exceptional clusters we get that every edge in $\mathcal{S}$ represents a $(2\eps, d/2)$-super-regular pair, and 
all non-exceptional clusters have the same cardinality, which is denoted by $m$. Moreover, $|W_0|\le 3dn$.
\end{lemma}

Since $v(G)-v(H)$ is bounded above by a constant, when embedding $H$ we need almost every vertex
of $G,$ in particular those in the exceptional cluster $W_0$. For this reason we will assign the vertices of $W_0$
to the stars in $\mathcal{S}$. This is not done in an arbitrary way. 

\begin{definition}\label{assign} Let $v\in W_0$ be a vertex and $(Q, T)$ be an $\eps$-regular pair. We say that \emph{$v\in T$ has large degree to} $Q$ if 
$v$ has at least $\eta |Q|/4$ neighbors in $Q$. Let $S=(A, B_1, \ldots, B_k)$ be a star in $\mathcal{S}$ where $A$ is the center of $S$ and $B_1, \ldots, B_k$ are the leaves, here $1\le k\le q$. 
If $v$ has large degree to any of $B_1, \ldots, B_k,$ then
$v$ \emph{can be assigned to} $A$. 
If $k<q$ and $v$ has large degree to $A,$ then $v$ \emph{can be assigned to} any of the $B_i$ leaves.
\end{definition}

\medskip

\begin{observation} If we assign new vertices to a $q$-star, then we necessarily assign them to the center. Since before assigning, the number of vertices in the leaf-clusters is exactly $q$ times the number of vertices in the center cluster, after assigning new vertices to the star, $q$ times the cardinality of the center will be larger than the total number of vertices in the leaf-clusters.
If $S\in \mathcal{S}$ is a $k$-star with $1\le k<q,$ and we assign up to $cm$ vertices to any of its clusters, where $0<c\ll 1,$ then even after assigning new vertices we will have that $q$ times the cardinality of the center is larger than the total number of vertices in the leaf-clusters. 
\end{observation}
 
The following lemma plays a crucial role in the embedding algorithm.

\begin{lemma}\label{szetosztas}
Every vertex of $W_0$ can be assigned to at least $\eta \ell/4$ non-exceptional clusters.
\end{lemma}

\begin{proof}
Suppose that there exists a vertex $w\in W_0$ that can be assigned to less than $\eta \ell/4$ clusters. 
If $w$ cannot be assigned to any cluster of some $k$-star $S_k$ with $k<q,$ then the total degree of $w$ into the clusters of $S_k$ is at most $k\eta m/4$.
If $w$ cannot be assigned to any cluster of some $q$-star $S_q,$ then the total degree of $w$ into the clusters of $S_q$ is at most $m+q\eta m/4,$ since
every vertex of the center cluster could be adjacent to $w$. Considering that $w$ can be assigned to at most $\eta \ell/4-1$ clusters and that $d(w, W-W_0)\ge n/(q+1)+\eta n/2,$ we
obtain the following inequality:

\begin{gather*}
	\frac{n}{q+1}+\frac{\eta n}{2}\le d(v, W-W_0)\le \eta \frac{\ell m}{4}+\\
\sum_{k=1}^{q-1}(k+1)\eta \frac{\ell_{k+1}m}{4}+ q\eta\frac{\ell_{q+1}m}{4}+\ell_{q+1}m.
\end{gather*}

Using  $m\ell \le n$ and  $\sum_{k=1}^q(k+1)\allowbreak\ell_{k+1}=\allowbreak\ell,$ we get

\begin{gather*}
\frac{m\ell}{q+1}+\frac{\eta m\ell}{2}\le \eta \frac{\ell m}{4}+(\ell-\ell_{q+1})\frac{\eta m}{4}+\\
 q\eta\frac{\ell_{q+1}m}{4}+\ell_{q+1}m.
\end{gather*}
Dividing both sides by $m$ and cancellations give
$$\frac{\ell}{q+1}\le q\frac{\eta \ell_{q+1}}{4}+(1-\frac{\eta}{4})\ell_{q+1}.$$ Noting that $(q+1)\ell_{q+1}\le \ell,$ one can easily see that we arrived at a contradiction. 
Hence every vertex of $W_0$ can be assigned to several non-exceptional clusters. 
\end{proof}

\medskip

\autoref{szetosztas} implies the following:

\begin{lemma}\label{szetosztas2}
One can assign the vertices of $W_0$ so that at most $\sqrt{d}m$ vertices are assigned to non-exceptional clusters. 
\end{lemma}

\begin{proof}
Since we have at least $\eta \ell/4$ choices for every vertex, the bound follows from the inequality $\frac{4|W_0|}{\eta \ell}\le \sqrt{d}m,$ where we used $d\ll \eta$ and $|W_0|\le  3dn$.
\end{proof}

\medskip 

\begin{observation}\label{fokok}
A key fact is that the number of newly assigned vertices to a cluster is much smaller than their degree into the opposite cluster of the regular pair since $\sqrt{d}m \ll \eta m/4$.
\end{observation}

\subsection{The embedding algorithm}

The embedding is done in two phases. In the first phase we cover every vertex that belonged to $W_0,$ together with some other vertices of the non-exceptional clusters.
In the second phase we are left with super-regular pairs into which we embed what is left from $H$ using the Blow-up lemma.

\subsubsection{The first phase}

Let $(A, B)$ be an $\eps$-regular cluster-edge in the $h$-star $\mathcal{C}\in \mathcal{S}$. 
We begin with partitioning $A$ and $B$ randomly, obtaining $A=A'\cup A''$ and $B=B'\cup B''$ with $A'\cap A''=B'\cap B''=\emptyset$. For every $w\in A$ (except those that came from $W_0$) flip a coin. 
If it is heads, we put $w$ into $A',$ otherwise we put it into $A''$. Similarly, we flip a coin for every $w\in B$  (except those that came from $W_0$) and depending on the outcome, we either put the 
vertex into $B'$ or into $B''$. The proof 
of the following lemma is standard, uses Chernoff's bound (see in~\cite{AS}), we omit it.

\begin{lemma}\label{random}
With high probability, that is, with probability at least $1-1/n,$ we have the following:

\begin{itemize}

\item $\big||A'|-|A''|\big|=o(n)$ and $\big||B'|-|B''|\big|=o(n)$

\item $deg(w, A'), deg(w, A'')>deg(w, A)/3$ for every $w\in B$

\item $deg(w, B'), deg(w, B'')>deg(w, B)/3$ for every $w\in A$

\item the density $d(A', B')\ge d/2$ 

\end{itemize}

\end{lemma}

it is easy to see that \autoref{random} implies that $(A', B')$ is a $(5\eps, d/6)$-super-regular pair having density at least $d/2$ with high probability.

Assume that $v$ was an element of $W_0$ before we assigned it to the cluster $A,$ and assume further
that $deg(v, B)\ge \eta m/4$. Since $(A, B)$ is an edge of the star-decomposition, either $A$ or $B$ must be the center of $\mathcal{C}$.

Let $H_i$ be one of the $q$-unbalanced bipartite subgraphs of $H$ that has not been embedded yet. We will use $H_i$ to cover $v$. Denote $S_i$ and $T_i$ the vertex classes
of $H_i,$ where $|S_i|\ge q|T_i|$. Let $S_i=\{x_1, \ldots,x_s\}$ and $T_i=\{y_1, \ldots, y_t\}$.

If $A$ is the center of $\mathcal{C}$ then the vertices of $T_i$ will cover vertices of $A',$ and the vertices of $S_i$ will cover vertices of $B'$. If $B$ is the center, $S_i$ and $T_i$ will switch roles.
The embedding of $H_i$ is essentially identical in both cases, so we will only discuss the case when $A$ is the center.\footnote{Recall that if $h<q$ then we may assigned $v$ to a leaf, so in such a case $B$ could be the center.} 

In order to cover $v$ we will essentially use a well-known method called Key lemma in~\cite{KS93}. We will heavily use the
fact that $$0<\eps \ll d \ll \eta.$$
The details are as follows.
We construct an edge-preserving injective mapping $\phi: S_i\cup T_i \longrightarrow A'\cup B'$. In particular, we
will have $\phi(S_i)\subset B'$ and $\phi(T_i)-v \subset A'$. First we let $\phi(y_1)=v$. Set $N_1=N(v)\cap B'$.
Using \autoref{random} we have that $|N_1|\ge \eta m/12\gg \eps m$. 

Next we find $\phi(y_2)$.  Since $|N_1|\gg \eps m,$ by $5\eps$-regularity the majority of the vacant vertices of $A'$
will have at least $d|N_1|/3$ neighbors in $N_1$. Pick any of these, denote it by $v_2$ and let $\phi(y_2)=v_2$.
Also, set $N_2=N_1\cap N(v_2)$. 

In general, assume that we have already found the vertices $v_2, v_3, \ldots, v_i,$ their common neighborhood
in $B'$ is $N_i,$ and \[|N_i|\ge \frac{\eta d^{i-1}}{3^{i-2}\cdot 36}m\gg \eps m.\]
By $5\eps$-regularity, this implies that the majority of the vacant vertices of $A'$ has large degree into $N_i,$ at least $d\cdot |N_i|/3,$ and this, as above, can be used
to find $v_{i+1}$. Then we set $\phi(y_{i+1})=v_{i+1}$. Since $\eta$ and $d$ is large compared to $\eps,$ even into
the last set $N_{t-1}$ many vacant vertices will have large degrees.

As soon as we have $\phi(y_1), \ldots, \phi(y_t),$ it is easy to find the images for $x_1, \ldots, x_t$. Since
$|N_t|\gg \eps m \gg s=|S_i|,$ we can arbitrarily choose $s$ vacant points from $N_t$ for the $\phi(x_j)$ images.

Note that we use less than $v(H_i)\le 4D^2$ vertices from $A'$ and $B'$ during this process. We can repeat it for
every vertex that were assigned to $A,$ and still at most $\sqrt{d}2D^2 m$ vertices will be covered from $A'$ and
from $B'$. 

Another observation is that every $h$-star in the decomposition before this embedding phase was $h$-unbalanced,
now, since we were careful, these have become $h'$-balanced with $h'\le h$. 

Of course, the above method will be repeated for every $(A, B)$ edge of the decomposition for which we have assigned
vertices of $W_0$.

\subsubsection{The second phase}

 In the second phase we first unite all the randomly partitioned clusters. For example, assume that after covering
the vertices that were coming from $W_0$ the set of vacant vertices of $A'$ is denoted by $A'_v$. Then we let
$A_v=A'_v\cup A'',$ and using analogous notation, let $B_v=B'_v\cup B''$. 

\begin{claim}\label{epszreg}
All the $(A_v, B_v)$ pairs are $(3\eps, d/6)$-super-regular with density at least $d/2$.
\end{claim}

\begin{proof}
The $3\eps$-regularity of these pairs is easy to see, like the lower bound for the density, since we have only covered
relatively few vertices of the clusters. For the large minimum
degrees note that by \autoref{random} every vertex of $A$ had at least $dm/6$ neighbors in $B'',$ hence, in $B_v$ as 
well, and analogous bound holds for vertices of $B$.  
\end{proof}

At this point we want to apply the Blow-up lemma for every star of $\mathcal{S}$ individually. For that we first have to assign those
subgraphs of $H$ to stars that were not embedded yet. We need a lemma.

\begin{lemma}\label{vegre}
Let $K_{a, b}$ be a complete bipartite graph with vertex classes $A$ and $B,$ where $|A|=a$ and $|B|=b$. Assume that
$a\le b= h a,$ where $1\le h\le q$. 
Let $H'$ be the vertex disjoint union of $q$-unbalanced bipartite graphs: 
\[H'=\bigcup\limits_{j=1}^t H_j,\]
 such that $v(H_j)\le 2D^2$ for every $j$. If $v(H') \le a+b-4(2q+1)D^2,$ then $H'\subset K_{a, b}$.
\end{lemma}

\medskip

Observe that if we have \autoref{vegre}, we can distribute the $H_i$ subgraphs among the stars of $\mathcal{S},$ and then apply the Blow-up lemma. Hence, we are done with proving \autoref{konstans} if we prove \autoref{vegre} above.

\begin{proof}
The proof is an assigning algorithm and its analysis.
We assign the vertex classes of the $H_j$ subgraphs to $A$ and $B,$ one-by-one. Before assigning the $j$th subgraph $H_j$
the number of vacant vertices of $A$ is denoted by $a_j$ and the number of vacant vertices of $B$ is denoted by $b_j$.

Assume that we want to assign $H_k$. If $ha_k-b_k>0,$ then the larger vertex class of $H_k$ is assigned to $A,$ the smaller
is assigned to $B$. Otherwise, if $ha_k-b_k\le 0,$ then we assign the larger vertex class to $B$ and the smaller one to $A$.
Then we update the number of vacant vertices of $A$ and $B$. Observe that using this assigning method we always have
$a_k\le b_k$.

The question is whether we have enough room for $H_k$. If $ha\ge 4hD^2,$ then we must have enough room, since
$b_k\ge a_k$ and every $H_j$ has at most $2D^2$ vertices. Hence, if the algorithm stops, we must have $a_k<4D^2$.
Since $b_k-ha_k\le 2D^2$ must hold, we have $b_k<(2h+1)2D^2<(2q+1)2D^2$. From this the lemma follows.    
\end{proof}

\section{Remarks}

One can prove a very similar result to \autoref{konstans}, in fact the result below follows easily from it.
For stating it we need  the notion of graph edit distance which is detailed e.g.~in~\cite{Ryan}: the edit distance 
between two graphs on the same labeled vertex set is defined
to be the size of the symmetric difference of the edge sets

\begin{theorem}
	\label{fotetel2}
		Let $q\ge 1$ be an integer. For every $\eta>0$ and $D \in\mathbb{N}$ there exists an $n_0 = n_0(\eta, q)$ and a $K = K(\eta, D, q)$ such that if $n\geq n_0$, $\pi$ is a $q$-unbalanced degree sequence of length $n$ with $\Delta(\pi)\leq D$, $G$ is a graph on $n$ vertices with $\delta(G)\geq\frac{n}{q+1}+\eta n$, then there exists a graph $G'$ on $n$ vertices such that the edit distance of $G$ and $G'$ is at most $K$, and $\pi$ can be embedded into $G'$.

\end{theorem}

Here is an example showing that \autoref{konstans} and \ref{fotetel2} are essentially best possible. 

\begin{example}
	Assume that $\pi$ has only odd numbers and $G$ has at least one odd sized component. The embedding is impossible. Indeed, any realization of $\pi$ has only even sized components, hence $G$ cannot contain it as a spanning subgraph.
\end{example}

Note that this example does not work in case $G$ is connected. In \autoref{fotetel} the minimum degree $\delta(G)\ge n/2+\eta n,$ hence, $G$ is connected, and in this case we can embed
$\pi$ into $G$.

%\bibliography{pakolas_osszevont}
%\bibliographystyle{abbrv}

\end{document}